\documentclass[10pt]{article}
\usepackage{amsthm, amsmath, amssymb,amsfonts}					
\usepackage{hyperref}
\usepackage{enumitem}

\newcommand{\bb}[1]{\mathbb{#1}}								
\newcommand{\ii}{\textup{i}}									
\newcommand{\mc}{\mathcal}

\newcommand{\sr}[1]{\rho\left(#1\right)}							
\newcommand{\vol}[1]{\operatorname{\rm Vol}\left( #1 \right)}				
\newcommand{\diam}{\text{diam}}

\newtheorem{theorem}{Theorem}
\newtheorem{lemma}[theorem]{Lemma}
\newtheorem{cor}[theorem]{Corollary}

\theoremstyle{definition}
\newtheorem{definition}[theorem]{Definition}
\newtheorem{rem}[theorem]{Remark}

\newtheorem{prob}[theorem]{Problem}

\title{The volume of the trace-nonnegative polytope via the Irwin-Hall Distribution}

\author{
Pietro Paparella\thanks{Division of Engineering and Mathematics, University of Washington Bothell, Bothell, Washington 98011, USA (pietrop@uw.edu)}
\and
Gregory K.~Taylor\thanks{College of William and Mary, Williamsburg VA 23186, USA (gktaylor@email.wm.edu)}
}

\begin{document}
\maketitle

\begin{abstract}
In this work, we find an explicit expression for the volume of the \emph{trace nonnegative polytope} via a generalization of the \emph{Irwin-Hall distribution}. This volume is an upper bound for the volume of all projected, normalized realizable spectra. We provide ancillary results on realizable trace-zero spectra and pose several problems suitable for further inquiry. 
\end{abstract} 

\section{Introduction and Background}

The \emph{real nonnegative inverse eigenvalue problem} (RNIEP) is to find necessary and sufficient conditions on $\sigma = \{ \lambda_1, \dots, \lambda_k \} \subset \bb{R}$ so that $\sigma$ is the spectrum of an entrywise-nonnegative matrix. If $A$ is a nonnegative matrix with spectrum $\sigma$, then $\sigma$ is called \emph{realizable} and $A$ is called a \emph{realizing matrix} for $\sigma$. Despite many stringent necessary conditions, the RNIEP remains unsolved when $k > 4$ (for background, see, e.g., \cite{eln2004}; for recent developments, see \cite{jp2015}). 

The set $\sigma = \{ \lambda_1,\dots,\lambda_k \} \subset \bb{R}$ is said to be \emph{normalized} if 
\[ \lambda_1 = 1 \geq \dots \geq \lambda_k. \] 
For a normalized set $\sigma$, let $x = x_\sigma = \begin{bmatrix} \lambda_2 & \dots & \lambda_k \end{bmatrix}^\top \in \bb{R}^{k-1}$. If $\mathcal{P}^{k-1}$ denotes the set of all projected $k$-tuples of all normalized spectra of nonnegative matrices, then  
\[ \mathcal{P}^{k-1} \subseteq \mathcal{T}^{k-1} := \left\{ x \in \bb{R}^{k-1}: ||x||_\infty \leq 1 \mbox{ and } 1+ \sum_{i=1}^{k-1} x_i \geq 0 \right\}. \] 
This follows from the Perron-Frobenius theorem and the fact that the realizing matrix is trace-nonnegative. The region $\mathcal{T}^{k-1}$, $k \geq 2$, is known as the \emph{trace nonnegative polytope} \cite{km2001}. It is well-known (see, e.g., \cite{jp2015, ll1978-79}) that $\mathcal{P}^{k-1} = \mathcal{T}^{k-1}$ for $2 \leq k \leq 4$.

The purpose of this work is to find an explicit expression for the volume of $\mathcal{T}^n$ ($n \geq 1$). The motivation is three-fold. First, it is clear that this is not a trivial endeavor: one approach is to enumerate the vertices of the polytope and slice it into simplices, at which point the formula for the volume of a simplex can be applied. However, enumerating these vertices is difficult (see, for e.g., \cite{mn2000}). Second, the volume of $\mathcal{T}^n$ gives an upper bound for the volume of $\mathcal{P}^{n}$ ($n \geq 1$). Lastly, in \cite{jp2015}, Johnson and Paparella studied polytopes whose points correspond to projected normalized spectra. In some cases the volume of these polytopes is available; thus, knowing the volume of $\mathcal{T}^n$ gives us a better impression of how ``big" these polytopes are.

In addition, we introduce a generalization of the \emph{Irwin-Hall distribution} which we call the \emph{$[a,b]$-uniform-sum distribution}, which, to the best of our knowledge, is not currently available in the literature (cf. \cite{bg2002}). We demonstrate that $\mathcal{P}^n \subset \mathcal{T}^n$, for every $n \geq 4$ (i.e., for spectra that contain at least five elements), and provide ancillary results on realizable trace-zero spectra. Finally, we pose the problem of finding an open set in the trace nonnegative polytope containing only non-realizable spectra. The discovery of such an open set would imply that the upper bound for the volume of the realizable region is indeed strict. Such an open set does not exist for $n \leq 4$ since the trace nonnegative polytope exactly coincides with the realizable region.

\section{A Generalization of the Irwin-Hall Distribution}

For $n \in \bb{N}$, let $\mathcal{B}^n := \left\{ x \in \bb{R}^n: ||x||_\infty \leq 1 \right\}$. For $i \in \{1,\dots, n\}$, let $Y_i \sim U[-1,1]$, and let $Y = \sum Y_i$. The fraction of the volume of $\mathcal{B}^n$ that coincides with $\mathcal{T}^n$ equals $P(Y \geq -1)$, i.e., 
\begin{equation*}
\vol{\mathcal{T}^n} = P(Y \geq -1) \vol{\mathcal{B}^n} = P(Y \geq -1)2^n. 
\end{equation*} 

For $i \in \{1,\dots, n\}$, let $X_i \sim U[0,1]$, and let $X = \sum X_i$. The continuous probability distribution for the random variable $X$ is the well-known \emph{Irwin-Hall} (or \emph{uniform-sum}) \emph{distribution} (IHD). The probability density function (PDF) $f$ of the IHD is given by
\begin{equation*}
f(x) = \frac{1}{(n-1)!} \sum_{k=0}^{\lfloor x \rfloor} (-1)^k {n \choose k} (x-k)^{n-1},
\end{equation*}
and the cumulative distribution function (CDF) $F$ is given by
\begin{equation*}
F(x) = \frac{1}{n!} \sum_{k=0}^{\lfloor x \rfloor} (-1)^k {n \choose k} (x-k)^n.
\end{equation*}
We refer to $X$ as the \emph{Irwin-Hall random variable}.  

To compute $ P(Y \geq -1)$, we need to generalize the IHD to capture the behavior of the sum of $n$ random variables uniformly distributed on the interval $[a,b]$. The generalization amounts to an affine transformation of an Irwin-Hall random variable. The theory of transformations on random variables is well-established (see, e.g., \cite{leemis2011,wackerly2007mathematical}). 

If $U$ and $V$ are random variables where $U = h(V)$ for some differentiable function $h$, then
$$f_U(u) = f_V[h^{-1}(u)] \left| \frac{d[h^{-1}(u)]}{du} \right|.$$
Integrating the pdf gives the following formula for the cdf
$$F_U(u) = F_V[h^{-1}(u)] .$$

For $i \in \{1,\dots, n\}$, let $Y_i \sim U[-1,1]$, and let $Y = \sum Y_i$. Let $X$ be as above. Note that $Y$ is an affine transformation of $X$ since
\begin{equation*}
Y = h(X) = |b-a|X + na.
\end{equation*}
This is clear when considering the support of each random variable: the support of $X$, denoted by $\mc{S}_X$, is the interval $[0,n]$; and the support of $Y$, denoted $\mc{S}_Y$, is the interval $[na, nb]$. From this, we see that $\diam(\mc{S}_Y) = n|b-a| = \diam(\mc{S}_X)|b-a|$. Also, the leftmost side of $\mc{S}_Y$ lies $na$ units from $0$. Both $X$ and $Y$ are identically distributed within their respective intervals. 

We can apply the general formula for the transformation of a random variable to derive the PDF and CDF of $Y$ in terms of the PDF and CDF of $X$:

\begin{equation}
\label{IHtransPDF}
f_Y(y) = \frac{1}{|b-a|} f_X\left(\frac{y-na}{|b-a|}\right)
\end{equation}
\begin{equation}
\label{IHtransCDF}
F_Y(y) =  F_X \left( \frac{y-na}{|b-a|} \right)
\end{equation}

These calculations lead to a generalization of the IHD, which we refer to as the \emph{$[a,b]$-uniform-sum distribution}.

\begin{theorem}
\label{abDef}
If $X_i \sim U[a,b]$, for $i =1,\dots,n$, and $X = \sum X_i$, then the probability density function and cumulative distribution function of $X$ are given by
\begin{equation}\label{abPDF}
f_X(x) = \frac{1}{|b-a|(n-1)!}\sum_{k=0}^{\lfloor h^{-1}(x) \rfloor} (-1)^k {n \choose k} \left[ h^{-1}(x)-k \right]^{n-1}
\end{equation}
and
\begin{equation}
F_X(x) = \frac{1}{n!} \sum_{k=0}^{\lfloor h^{-1}(x) \rfloor} (-1)^k {n \choose k} \left[ h^{-1}(x)-k \right]^n, 
\end{equation}\label{abCDF}
respectively, where $h^{-1}(x) = \frac{x-na}{|b-a|}$.
\end{theorem}

We are now able to give an expression for the volume of $\mathcal{T}^n$.

\begin{cor}
\label{volumecor}
The volume of the $n$-dimensional trace-nonnegative polytope $\mathcal{T}^n$ is given by 
\begin{equation*}
Vol(\mathcal{T}^n) = 2^n \left[1 - \frac{1}{n!}\sum_{k=0}^{\lfloor \frac{n-1}{2} \rfloor} (-1)^k{n \choose k} \left(\frac{n-1}{2}-k \right)^n \right].
\end{equation*}
\end{cor}

\begin{proof}
Follows from $Vol(\mathcal{T}^n) = Vol(\mathcal{B}^n) P(X \geq -1) = 2^n(1-F_X(-1))$.
\end{proof}

\section{Non-Realizable Spectra within the Trace Nonnegative Polytope}

In this section, we provide non-realizable spectra within the trace nonnegative polytope for all $n \geq 5$. This generalizes a well-known example of such a spectrum for $n=5$ given by Friedland in \cite{friedland1978inverse}.

If $\sigma = \{ \lambda_1,\dots, \lambda_n \}$ is realizable, then
\begin{equation}
\sr{\sigma} := \max_{i} |\lambda_i| \in \sigma	\label{specrad}
\end{equation}
and 
\begin{equation} 
s_1(\sigma) := \sum_{i=1}^n \lambda_i  \geq 0 	\label{tracecond}. 	
\end{equation}
As mentioned in the introduction, it is well-known that for $1\leq n \leq 4$, conditions \eqref{specrad} and \eqref{tracecond} are also sufficient for realizability (see, e.g., \cite{jp2015, ll1978-79}). 

For $n = 5$, the normalized trace-zero spectrum 
\begin{equation*}
\sigma =  \{ 1, 1, -2/3, -2/3, -2/3 \}
\end{equation*}
is not realizable \cite{friedland1978inverse}. Indeed, if $\sigma$ is realizable, then the realizing matrix must be reducible. Thus, there is a partition $(\sigma_1,\sigma_2)$ of $\sigma$  such that each $\sigma_i$ satisfies \eqref{specrad} and \eqref{tracecond}. Clearly, this is impossible. 

As the next result shows, this construction generalizes to all odd orders greater than or equal to five.

\begin{theorem}
\label{oddNRSpectra}
Let $n = 2k+1$ for some integer $k \geq 2$. If 
\begin{equation*}
\sigma_n := \{ \overbrace{1, \dots, 1}^{k}, \overbrace{-k/(k+1), \dots, -k/(k+1)}^{k+1} \}, 
\end{equation*}
then $s_1(\sigma) = 0$ and $\sigma_n$ is not realizable. 
\end{theorem}

\begin{proof}
We proceed by contradiction. If $\sigma_n$ is realizable, then the realizing matrix must be reducible. Thus, there exists a partition $(\sigma_1,\dots,\sigma_k)$ of $\sigma_n$ such that each $\sigma_i$ satisfies \eqref{specrad} and \eqref{tracecond}. The impossibility of such a partition is guaranteed by the \emph{pigeon-hole principle} (see, e.g., \cite[Theorem 3.1.1]{b2010}). 
\end{proof}

Now, we establish a similar construction for even orders. Consider the spectrum $\sigma = \{ 1, 1, -1/5, -3/5, -3/5, -3/5 \}$. If $\sigma$ is realizable, then the realizing matrix must be reducible. Thus, there is a partition $(\sigma_1,\sigma_2)$ of $\sigma$ such that $\sigma_i$ satisfies \eqref{specrad} and \eqref{tracecond}. This partition is impossible because $\{1,-3/5, -3/5 \}$ must be a subset of either $\sigma_1$ or $\sigma_2$. 

We can generalize this construction to all even orders greater than or equal to six. 

\begin{theorem}
\label{evenNRSpectra}
Let $n = 2(k+1)$ for some integer $k \geq 2$. If 
\begin{equation*}
\sigma_n := \left\{ \overbrace{1, \dots, 1}^k, \frac{-1}{2k+1}, \overbrace{\frac{1-2k}{2k+1}, \dots, \frac{1-2k}{2k+1}}^{k+1} \right\},
\end{equation*}
then $s_1(\sigma) = 0$ and $\sigma_n$ is not realizable. 
\end{theorem}

\begin{proof}
We proceed by contradiction. If $\sigma_n$ is realizable, then the realizing matrix must be reducible. Thus, there is a partition $(\sigma_1,\dots,\sigma_k)$ of $\sigma_n$ such that each $\sigma_i$ satisfies \eqref{specrad} and \eqref{tracecond}. Following the \emph{pigeon-hole principle} (see, e.g., \cite[Theorem 3.1.1]{b2010}), 
\begin{equation*}
\left\{ 1,\frac{1-2k}{2k+1}, \frac{1-2k}{2k+1} \right\} \subseteq \sigma_i,
\end{equation*} but $s_1 (\sigma) \leq 1 +2(1-2k)/(2k+1) = (3-2k)/(2k+1) < 0$, a contradiction.
\end{proof}

Since $\mathcal{P}^n \subset \mathcal{T}^n$, it follows that $\vol{\mathcal{P}^n} \leq \vol{\mathcal{T}^n}$ and it is natural to consider whether this inequality is strict. This can be settled by investigating the following nontrivial problem. 

\begin{prob}
Determine whether $\mathcal{T}^n \backslash \mathcal{P}^n$ contains an open-set.
\end{prob}

\section{A Characterization of Partitionable Trace-Zero Spectra}

Here, we present a necessary and sufficient condition on the realizability of certain trace-zero spectra which generalizes the notions from Theorems \hyperref[oddNRSpectra]{\ref*{oddNRSpectra}} and \hyperref[evenNRSpectra]{\ref*{evenNRSpectra}}. First, we prove the following lemma.

\begin{lemma} 
\label{taylorlemma}
Let $\sigma$ be a realizable spectrum. If $\sigma = \sigma_1 \cup \dots \cup \sigma_k$, where each $\sigma_i$ is realizable, then $s_1(\sigma_i) \leq s_1(\sigma)$.
\end{lemma}

\begin{proof}
Clear given that $s_1(\sigma) \geq 0$ and $s_1(\sigma) = \sum_{i=1}^k s_1(\sigma_i)$. 
\end{proof}

This is a rather simple idea, but it proves to be very useful when considering the realizability of trace-zero spectra. Before we state the main result, we introduce a certain type of spectrum.

\begin{definition}
{\rm A normalized spectrum $\sigma$ is called a \emph{Sule{\u{\i}}manova spectrum} if $s_1(\sigma) \geq 0$ and the only positive eigenvalue is 1.}
\end{definition}

\begin{rem}\label{sulrem}
{\rm Friedland \cite{friedland1978inverse} and Perfect \cite{p1953} proved that every Sule{\u{\i}}manova spectrum is realizable via companion matrices (for other proofs, see references in \cite{friedland1978inverse}). Recently, Paparella \cite{p2015} gave a constructive proof via \emph{permutative matrices}.}
\end{rem}

\begin{theorem}
Let $\sigma = \{ \lambda_1, \dots, \lambda_n \} \subset \bb{R}$ and suppose that $\sigma$ satisfies \eqref{specrad}; \eqref{tracecond}; $| \sr{\sigma} \cap \sigma| = k > 1$; and $\lambda < 0$ for every $\lambda \not \in \sr{\sigma} \cap \sigma$. Then $\sigma$ is realizable if and only if it is the union of Sule{\u{\i}}manova spectra.
\end{theorem}

\begin{proof}
We can assume that $\sigma = \{1, \dots, 1, \lambda_{k+1}, \dots, \lambda_n \}$, where $\lambda_i < 0$ for $i \in \{k+1,\dots,n\}$. Thus, there exists a partition $(\sigma_1,\dots,\sigma_k)$ of $\sigma$ such that each $\sigma_i$ is realizable. Lemma \hyperref[taylorlemma]{\ref*{taylorlemma}} implies that $s_1(\sigma_i) \geq 0$. Because $\lambda_i < 0$, if $\lambda_i \neq 1$, each $\sigma_i$ is a Sule{\u{\i}}manova spectrum. 

The converse follows from Remark \hyperref[sulrem]{\ref*{sulrem}}.
\end{proof}

\section{Conclusion}

We conclude by introducing a generalization of our original problem. The \emph{nonnegative inverse eigenvalue problem} is to determine necessary and sufficient conditions such that $\sigma = \{ \lambda_1, \dots, \lambda_n\} \subset \bb{C}$ is the spectrum of a nonnegative matrix. In addition to satisfying \eqref{specrad} and \eqref{tracecond}, $\sigma$ must be \emph{self-conjugate}, i.e., $\bar{\lambda} \in \sigma$ if $\lambda \in \sigma$.

Without loss of generality, we may write 
\begin{equation*}
\sigma = \{ 1, \lambda_1, \dots, \lambda_r, \mu_1 \pm \nu_1 \ii,\dots, \mu_c \pm \nu_c \ii \},  
\end{equation*}
where
\begin{enumerate}[label=(\roman*)]
\item $\Im({\lambda_i}) = 0$, for all $i \in \{1,\dots,r\}$; 
\item $|\lambda_i| \leq 1$, for all $i \in \{1,\dots,r\}$; 
\item $\nu_i \neq 0$, for all $i \in \{1,\dots,c\}$; and
\item $|\mu_i + \nu \ii | = \sqrt{\mu_i^2 + \nu_i^2} \leq 1$ for all $i \in \{1,\dots,c\}$.
\end{enumerate}
With the above in mind, \eqref{tracecond} can be written as 
\begin{equation*}
1 + \sum_{i=1}^r \lambda_i + 2 \sum_{i=1}^c \mu_i \geq 0.
\end{equation*}

Let $\mathcal{B}_\infty^n := \{ x \in \mathbb{R}^n : \begin{Vmatrix} x \end{Vmatrix}_\infty \leq 1 \}$ and $\mathcal{B}_2^n := \{ x \in \mathbb{R}^n : \begin{Vmatrix} x \end{Vmatrix}_2 \leq 1 \}$. We identify $a + b \ii$ with $(a,b) \in \mathbb{R}^2$. \newline

\begin{prob} 
{\rm For $n \geq 1$, find the volume of the \emph{trace-nonnegative region}
\begin{equation*}
\mathcal{TN}_{n} := 
\left\{ 
(\lambda_1,\dots, \lambda_r) \times (\mu_1,\nu_1,\dots, \mu_c,\nu_c) \in \mathcal{B}_\infty^r \times \mathcal{B}_2^{2c} : 
1 + \sum_{i=1}^r \lambda_i + 2 \sum_{i=1}^c \mu_i \geq 0 
\right\}
\end{equation*}}
\end{prob}

\begin{rem}
Corollary \hyperref[volumecor]{\ref*{volumecor}} solves Problem 1 when $c = 0$.
\end{rem}

\bibliographystyle{abbrv}
\bibliography{refs}

\end{document}